\documentclass[12pt]{article}
\usepackage{amsmath,amssymb,amsthm, comment}

\newtheorem{theorem}{Theorem}[section]
\newtheorem{thmy}{Theorem}

\newtheorem{lemma}[theorem]{Lemma}
\newtheorem{corollary}[theorem]{Corollary}

\def\barr{\begin{array}}
\def\earr{\end{array}}

\title{Another criterion for supersolvability\\ of finite groups}
\author{Marius T\u arn\u auceanu}
\date{January 13, 2022}

\begin{document}

\maketitle

\begin{abstract}
Let $o(G)$ be the average order of a finite group $G$. In this paper, we prove that if $o(G)<\frac{31}{12}$\,, then $G$ is
supersolvable. Moreover, we have $o(G)=\frac{31}{12}$ if and only if $G\cong A_4$. We also classify finite groups $G$
satisfying $o(G)<\frac{31}{12}$\,.
\end{abstract}

{\small
\noindent
{\bf MSC2000\,:} Primary 20D60; Secondary 20D10, 20F16.

\noindent
{\bf Key words\,:} average order, sum of element orders, supersolvable group.}

\section{Introduction}
Given a finite group $G$, we denote by $\psi(G)$ the sum of element orders of $G$ and by $o(G)$ the average order of $G$, that is
\begin{equation}
\psi(G)=\sum\limits_{x\in G}o(x) \mbox{ and } o(G)=\frac{\psi(G)}{|G|}\,.\nonumber
\end{equation}

In the last years there has been a growing interest in studying the pro\-per\-ties of these functions and their relations with the structure of $G$ (see for example \cite{1}-\cite{4}, \cite{6}-\cite{8}, \cite{10}-\cite{11}, \cite{14} and \cite{18}-\cite{19}).

In \cite{10}, A. Jaikin-Zapirain uses the average order to determine a lower bound for the number of conjugacy classes of a finite $p$-group/nilpotent group. He also suggests the following question: "\textit{Let $G$ be a finite {\rm(}$p$-{\rm)}group and $N$ be a normal {\rm(}abelian{\rm)} subgroup of $G$. Is it true that $o(G)\geq o(N)^{\frac{1}{2}}$}\,\,?\,". Recently, E.I. Khukhro, A. Moret\' o and M. Zarrin proved the following result (see Theorem 1.2 of \cite{11}):

\begin{thmy}
Let $c>0$ be a real number and $p\geq \frac{3}{c}$ be a prime. Then there exists a finite $p$-group $G$ with a normal abelian subgroup $N$ such that $o(G)<o(N)^c$.
\end{thmy}

Note that Theorem A provides a negative answer to Jaikin-Zapirain's question even if we replace the exponent $\frac{1}{2}$ with any positive real number $c$. In the same paper \cite{11}, the authors posed the following conjecture:

\bigskip\noindent{\bf Conjecture.} {\it Let $G$ be a finite group and suppose that $o(G)<\frac{211}{60}=o(A_5)$. Then $G$ is solvable.}
\bigskip

This has been confirmed by M. Herzog, P. Longobardi and M. Maj \cite{8}.

\begin{thmy}
Let $G$ be a finite group. If $o(G)<\frac{211}{60}$\,, then $G$ is solvable. Moreover, we have $o(G)=\frac{211}{60}$ if and only if $G\cong A_5$.
\end{thmy}

Inspired by these results, we came up with the following new criterion for supersolvability of finite groups.

\begin{theorem}
Let $G$ be a finite group. If $o(G)<\frac{31}{12}$\,, then $G$ is supersolvable. Moreover, we have $o(G)=\frac{31}{12}$ if and only if $G\cong A_4$.
\end{theorem}

Theorem 1.1 also leads to a classification of finite groups $G$ with $o(G)<\frac{31}{12}$\,, modulo $2$-groups.

\begin{theorem}
Let $G$ be a finite group satisfying $o(G)<\frac{31}{12}$ and $n_2(G)$ be the number of elements of order $2$ in $G$. Then one of the following statements holds:
\begin{itemize}
\item[{\rm a)}] $G$ is a $2$-group with $G'=\Phi(G)$ and $n_2(G)>\frac{17}{24}\,|G|-\frac{3}{2}$\,;
\item[{\rm b)}] $G\cong C_3$;
\item[{\rm c)}] $G\cong K\rtimes H$ is a Frobenius group whose kernel $K$ is an elementary abelian $3$-group and the complement $H$ is cyclic of order $2$; moreover, in this case we have
\begin{equation}
o(G)=\frac{5\cdot 3^m-2}{2\cdot 3^m}\,,\nonumber
\end{equation}where $|K|=3^m$.
\end{itemize}
\end{theorem}
\smallskip

For the proof of our results, we need the following two theorems. Recall that a \textit{just non-supersolvable group} is a solvable group which is not supersolvable, but all of whose proper quotients are supersolvable.

\begin{thmy}
{\rm (D.J.S. Robinson and J.S. Wilson \cite{15})} Let $G$ be a finite just non-supersolvable group. Then $G$ splits over its Fitting subgroup $A$ and all complements of $A$ are conjugate. Moreover, $A$ is abelian and noncyclic, $Q=G/A$ is supersolvable, and $A$ is faithful and simple as a $Q$-module.

Conversely, any extension by a finite supersolvable group $Q$ of a faithful simple $Q$-module which is not $\mathbb{Z}$-cyclic is a finite just non-supersolvable group.
\end{thmy}

\begin{thmy}
Let $G$ be a finite group, $\pi_e(G)$ be the set of element orders of $G$ and $n_d(G)$ {\rm($s_d(G)$)} be the number of elements {\rm(}subgroups{\rm)} of order $d$ in $G$, $\forall\, d\in\mathbb{N}$. Then the following statements hold:
\begin{itemize}
\item[{\rm 1)}]{\rm (T.J. Laffey \cite{12})} If $p$ is a prime divisor of $|G|$ and $G$ is not a $p$-group, then
\begin{equation}
n_p(G)\leq\frac{p}{p+1}\,|G|-1.\nonumber
\end{equation}
\item[{\rm 2)}]{\rm (T.J. Laffey \cite{13})} If $G$ is a $3$-group and $\exp(G)\neq 3$, then
\begin{equation}
n_3(G)\leq\frac{7}{9}\,|G|-1.\nonumber
\end{equation}
\item[{\rm 3)}]{\rm (R. Brandl and W. Shi \cite{5})} If $\pi_e(G)=\{1,2,3\}$, then $G\cong K\rtimes H$ is a Frobenius group with kernel $K$ and complement $H$, where either $K\cong C_3^m$ and $H\cong C_2$, or $K\cong C_2^{2m}$ and $H\cong C_3$.
\item[{\rm 4)}]{\rm (M. T\u arn\u auceanu \cite{17})} If $G$ is a $2$-group of order $2^n$ and $\exp(G)\neq 2$, then
\begin{equation}
s_{2^k}(G)\leq s_{2^k}(D_8\times C_2^{n-3}),\, \forall\, k=0,1,\ldots,n.\nonumber
\end{equation}
\end{itemize}
\end{thmy}
\smallskip

We will also use the next two basic properties of the function $o(G)$:

\begin{itemize}
\item[] - it is multiplicative, that is if $(G_i)_{i=\overline{1,k}}$ are finite groups of coprime orders, then
\begin{equation}
o(\prod_{i=1}^kG_i)=\prod_{i=1}^ko(G_i);
\end{equation}
\item[] - if $G$ is a finite group and $X$ is a non-trivial normal subgroup of $G$, then
\begin{equation}
o\left(\frac{G}{X}\right)<o(G).
\end{equation}
\end{itemize}

Most of our notation is standard and will usually not be repeated here. Elementary notions and results on groups can be found in \cite{9,16}.

\section{Proofs of the main results}

We start with the following easy but important lemma.

\begin{lemma}
Let $G$ be a finite group, $1=d_1<d_2<\ldots <d_r$ be the element orders of $G$ and $n_{d_i}(G)$ be the number of elements of order $d_i$ in $G$, $\forall\, i= 1,2,\ldots,r$. Assume that $r\geq 3$ and take a positive integer $s$ with $3\leq s\leq r$. If $o(G)<c$, where $c>0$ is a real number, then
\begin{equation}
n_{d_{s-1}}(G)>\frac{d_s-c}{d_s-d_{s-1}}\,|G|-\sum_{i=1}^{s-2}\frac{d_s-d_i}{d_s-d_{s-1}}\,n_{d_i}(G).
\end{equation}
\end{lemma}

\begin{proof}
We have $|G|=1{+}n_{d_2}(G){+}\ldots{+}n_{d_s}(G){+}\ldots{+}n_{d_r}(G)$, so we deduce that
\begin{align*}
\psi(G)&=1{+}d_2n_{d_2}(G){+}\ldots{+}d_sn_{d_s}(G){+}\ldots{+}d_rn_{d_r}(G)\\
&\geq 1{+}d_2n_{d_2}(G){+}\ldots{+}d_{s-1}n_{d_{s-1}}(G){+}d_s(n_{d_s}(G){+}\ldots{+}n_{d_r})(G)\\
&=1{+}d_2n_{d_2}(G){+}\ldots{+}d_{s-1}n_{d_{s-1}}(G){+}d_s(|G|{-}1{-}n_{d_2}(G){-}\ldots{-}n_{d_{s-1}}(G)).
\end{align*}Since $o(G)<c$, it follows that $\psi(G)<c|G|$. Therefore we have
$$c|G|>1{+}d_2n_{d_2}(G){+}\ldots{+}d_{s-1}n_{d_{s-1}}(G){+}d_s(|G|{-}1{-}n_{d_2}(G){-}\ldots{-}n_{d_{s-1}}(G)).$$
Clearly, the last inequality is equivalent to (3), completing the proof.
\end{proof}

Our second lemma collects information about some particular classes of finite groups $G$ satisfying $o(G)<\frac{31}{12}$\,.

\begin{lemma}
Given a finite group $G$ such that $o(G)<\frac{31}{12}$\,, the following statements hold:
\begin{itemize}
\item[{\rm a)}] If $|G|$ is odd, then $G\cong C_3$.
\item[{\rm b)}] If $G$ is abelian, then either $G\cong C_3$ or $G$ is an elementary abelian $2$-group.
\item[{\rm c)}] If $G$ is a $2$-group, then $G'=\Phi(G)$ and $n_2(G)>\frac{17}{24}\,|G|-\frac{3}{2}$\,.
\item[{\rm d)}] If $G$ is supersolvable of even order but not a $2$-group, then $G\cong K\rtimes H$ is a Frobenius group whose kernel $K$ is an elementary abelian $3$-group and the complement $H$ is cyclic of order $2$; moreover, in this case we have\newpage
\begin{equation}
o(G)=\frac{5\cdot 3^m-2}{2\cdot 3^m}\,,\nonumber
\end{equation}where $|K|=3^m$.
\end{itemize}
\end{lemma}

\begin{proof}
\begin{itemize}
\item[{\rm a)}] It suffices to observe that if $|G|\geq 5$, then
\begin{equation}
o(G)\geq\frac{1+3(|G|-1)}{|G|}>\frac{31}{12}\,.\nonumber
\end{equation}
\item[{\rm b)}] Let $G=G_1\times G_2\times\cdots\times G_k$, where $G_i$ is an abelian $p_i$-group, $\forall\, i=1,2,\ldots,k$. By using (1), it follows that if $k\geq 2$, then
\begin{equation}
o(G)=\prod_{i=1}^k o(G_i)>\frac{31}{12}\,.\nonumber
\end{equation}So, we can assume that $k=1$, i.e. $G$ is an abelian $p$-group. For $p\geq 3$ we get $G\cong C_3$ from a), while for $p=2$ and $\exp(G)\neq 2$ we have
\begin{equation}
n_2(G)>\frac{4-\frac{31}{12}}{4-2}\,|G|-\frac{4-1}{4-2}=\frac{17}{24}\,|G|-\frac{3}{2}\nonumber
\end{equation}by (3). Then
\begin{equation}
|\Omega_1(G)|=n_2(G)+1>\frac{17}{24}\,|G|-\frac{1}{2}>\frac{1}{2}\,|G|.\nonumber
\end{equation}Since $\Omega_1(G)\leq G$, it follows that $\Omega_1(G)=G$, a contradiction. Thus $\exp(G)=2$, i.e. $G$ is an elementary abelian $2$-group.
\item[{\rm c)}] Since $\frac{G}{G'}$ is abelian and $o\!\left(\frac{G}{G'}\right)\leq o(G)<\frac{31}{12}$\,, from b) we deduce that $\frac{G}{G'}$ is elementary abelian and so $G'=\Phi(G)$. Obviously, the inequality $n_2(G)>\frac{17}{24}\,|G|-\frac{3}{2}$ holds if $G$ is elementary abelian. We observe that it also holds if $\exp(G)\neq 2$ by applying (3) for $s=3$.
\item[{\rm d)}] Since $G$ is supersolvable, we have $G\cong K\rtimes H$, where $K$ is the cha\-rac\-te\-ris\-tic subgroup consisting of all elements of odd order in $G$ and $H$ is a Sylow $2$-subgroup of $G$. Also, we know that $G$ has quotients isomorphic to $C_2$, which implies that $2\mid |\frac{G}{G'}|$ and therefore $\frac{G}{G'}\cong C_2^q$, where $q\in\mathbb{N}^*$. This leads to $K\subseteq G'$. Let $H_1$ be a complement of $K$ in $G'$. Then $G'=K\times H_1$ because $G'$ is nilpotent. It follows that $H_1$ is characteristic in $G$ and we have\newpage
\begin{equation}
\frac{G}{H_1}\cong K\rtimes\frac{H}{H_1}\cong K\rtimes C_2^q \mbox{ and } o\!\left(\frac{G}{H_1}\right)\leq o(G)<\frac{31}{12}\,.\nonumber
\end{equation}

\hspace{5mm}Let $\pi(G)=\{2=p_1,p_2,\ldots,p_r\}$ be the set of primes dividing $|G|$, where $p_1<p_2<\ldots<p_r$. By induction on $r$, we infer that $G$ has a quotient of type $K_r\rtimes C_2^q$, where $K_r$ is a Sylow $p_r$-subgroup of $G$. Take $K_{r1}\lhd K_r\rtimes C_2^q$ with $|K_{r1}|=p_r$. Then
\begin{equation}
\frac{K_r\rtimes C_2^q}{K_{r1}}\cong\frac{K_r}{K_{r1}}\rtimes C_2^q \mbox{ and } \left|\frac{K_r}{K_{r1}}\right|<|K_r|.\nonumber
\end{equation}By repeating this process, we get that $G$ has a quotient of type $C_{p_r}\rtimes C_2^q$. Now it is easy to see that the condition $o(C_{p_r}\rtimes C_2^q)<\frac{31}{12}$ implies $q=1$ and $p_r=3$. Thus $H_1$ is a maximal subgroup of $H$. Suppose that $H$ contains a maximal subgroup $M\neq H_1$. Then $KM$ is a normal subgroup of index $2$ in $G$ and so $\frac{G}{KM}\cong C_2$. This shows that $G'\subseteq KM$, which leads to $G'=KM$. Then $M\subseteq G'$ and therefore $H=MH_1\subseteq G'$, a contradiction. Consequently, $H_1$ is the unique maximal subgroup of $H$, i.e. $H$ is cyclic. Moreover, we remark that $H\cong C_2$ because
\begin{equation}
o(C_{2^t})=\frac{2^{2t+1}+1}{3\cdot 2^t}>\frac{31}{12}\,, \forall\, t\geq 2.\nonumber
\end{equation}Hence $G\cong K\rtimes H$, where $|K|=3^m$ and $H\cong C_2$.

\hspace{5mm}Since the first three element orders of $G$ are $d_1=1$, $d_2=2$ and $d_3=3$, by applying (3) for $s=3$ we obtain
\begin{equation}
n_2(G)>\frac{5}{12}\,|G|-2=\frac{5}{6}\,3^m-2.\nonumber
\end{equation}Assume that $n_2(G)\neq 3^m$. Then $n_2(G)\leq 3^{m-1}$ because $n_2(G)\mid 3^m$, which implies that
\begin{equation}
\frac{5}{6}\,3^m-2<3^{m-1}.\nonumber
\end{equation}This leads to $m=1$, i.e. $G\cong S_3$, a contradiction. Thus $n_2(G)=3^m$. We also observe that $G$ cannot have elements of order $6$.

\hspace{5mm}Assume that $\exp(K)\neq 3$. Then $d_4=9$ and from (3) for $s=4$ it follows that\newpage
\begin{align*}
n_3(K)=n_3(G)&>\frac{77}{72}\,|G|-\frac{4}{3}-\frac{7}{6}\,3^m\\
&=\frac{35}{36}\,3^m-\frac{4}{3}\\
&>\frac{7}{9}\,3^m-1,\nonumber
\end{align*}contradicting Theorem D, 2). Consequently, $\exp(K)=3$ and so $\pi_e(G)=\{1,2,3\}$. Thus $G$ is a Frobenius group with kernel $K\cong C_3^m$ and complement $H\cong C_2$ by Theorem D, 3). Also, it is clear that
\begin{equation}
o(G)=\frac{1+2\cdot 3^m+3(3^m-1)}{2\cdot 3^m}=\frac{5\cdot 3^m-2}{2\cdot 3^m}\,,\nonumber
\end{equation}completing the proof.
\end{itemize}
\end{proof}

The following consequence of Lemma 2.2 is immediate.

\begin{corollary}
A finite supersolvable group $G$ satisfying $o(G)<\frac{31}{12}$ is of one of the following types: a $2$-group as in c), $C_3$ or a semidirect product $C_3^m\rtimes C_2$ as in d).
\end{corollary}

We are now able to prove our main theorems.

\bigskip\noindent{\bf Proof of Theorem 1.1.} Assume the result is false and let $G$ be a coun\-ter\-exam\-ple of minimal order. Then every proper quotient of $G$ is supersolvable by (2). On the other hand, since
\begin{equation}
o(G)<\frac{31}{12}<\frac{211}{60}\,,\nonumber
\end{equation}Theorem B implies that $G$ is solvable. Thus $G$ is a just non-supersolvable group and its structure is given by Theorem C. More precisely, $G$ has a unique minimal normal subgroup $N\cong C_p^r$, where $p$ is a prime and $r\geq 2$. We also have $\Phi(G)=1$. Indeed, if $\Phi(G)\neq 1$, then $o(\frac{G}{\Phi(G)})<o(G)<\frac{31}{12}$ implies that $\frac{G}{\Phi(G)}$ is supersolvable and so $G$ itself is supersolvable, a contradiction. Since $N\nsubseteq\Phi(G)$, there is a maximal subgroup $M$ of $G$ such that $N\nsubseteq M$. Then $MN=G$ and $M\cap N\lhd G$. By the minimality of $N$, we get $M\cap N=1$, which shows that $M$ is a complement of $N$ in $G$. Moreover, $N$ coincides with the Fitting subgroup of $G$ and the supersolvable subgroup $M$ acts faithfully on $N$. Also, we have
\begin{equation}
o(M)=o(\frac{G}{N})<o(G)<\frac{31}{12}\,.\nonumber
\end{equation}

By Lemma 2.2, a), $|G|=p^r|M|$ is even and therefore we distinguish the following two cases.

\medskip
\hspace{5mm}\noindent{\bf Case 1.} $|M|$ is odd\vspace{2mm}

Then $M\cong C_3$ and $p=2$, that is $G=C_2^r\rtimes C_3$. Since $r\geq 2$, we get
\begin{equation}
o(G)=\frac{1+2(2^r-1)+3\cdot 2^{r+1}}{3\cdot 2^r}=\frac{2^{r+3}-1}{3\cdot 2^r}\geq\frac{31}{12}\,,\nonumber
\end{equation}a contradiction.

\medskip
\hspace{5mm}\noindent{\bf Case 2.} $|M|$ is even\vspace{2mm}

We have the next two subcases.

\medskip
\hspace{10mm}\noindent{\bf Subcase 2.1.} $M$ is a $2$-group\vspace{2mm}

Let $|M|=2^q$. Since $G$ is not supersolvable, it is not strictly $2$-closed and therefore $M\not\cong C_2^q$. We also remark that $p$ must be odd.

If $p\geq 5$, then the first three element orders of $G$ are $1$, $2$ and $4$. By applying (3) for $s=3$, it follows that $n_2(G)>\frac{17}{24}\,|G|-\frac{3}{2}$\,. On the other hand, Theorem D, 1), shows that $n_2(G)\leq\frac{2}{3}\,|G|-1$. Thus
\begin{equation}
\frac{17}{24}\,|G|-\frac{3}{2}<\frac{2}{3}\,|G|-1,\nonumber
\end{equation}i.e. $|G|<12$, contradicting the fact that $G$ is not supersolvable.

If $p=3$, then the first four element orders of $G$ are $1$, $2$, $3$, $4$, and the inequality (3) for $s=4$ leads to
\begin{equation}
n_3(G)>\frac{17}{12}\,|G|-3-2n_2(G).\nonumber
\end{equation}But $n_3(G)=n_3(N)=3^r-1$ and therefore
\begin{equation}
n_2(G)>\frac{17}{24}\,|G|-\frac{3^r}{2}-1.\nonumber
\end{equation}So, we have
\begin{equation}
\frac{17}{24}\,|G|-\frac{3^r}{2}-1<\frac{2}{3}\,|G|-1,\nonumber
\end{equation}i.e. $q\in\{1,2,3\}$, and the condition $o(M)<\frac{31}{12}$ implies that $M=D_8$. It is now easy to check that $\pi_e(G)=\{1,2,3,4\}$ and $n_1(G)=1$, $n_2(G)=5\cdot 3^r$, $n_3(G)=3^r-1$, $n_4(G)=2\cdot 3^r$. Thus
\begin{equation}
o(G)=\frac{1+10\cdot 3^r+3(3^r-1)+8\cdot 3^r}{8\cdot 3^r}=\frac{21\cdot 3^r-2}{8\cdot 3^r}>\frac{31}{12}\,,\nonumber
\end{equation}a contradiction.

\medskip
\hspace{10mm}\noindent{\bf Subcase 2.2.}  $M$ is not a $2$-group\vspace{2mm}

Then the structure of $M$ is given by Lemma 2.2, d), namely $M\cong K\rtimes H$ is a Frobenius group with kernel $K\cong C_3^m$ and complement $H\cong C_2$. Also, we observe that $p\neq 3$. Indeed, if $p=3$, then $G$ is strictly $2$-closed and consequently supersolvable, a contradiction.

Assume first that $p\geq 5$. The the inequality (3) for $s=3$ becomes
\begin{equation}
n_2(G)>\frac{5}{12}\,|G|-2.\nonumber
\end{equation}On the other hand, $n_2(G)$ is the number of Sylow $2$-subgroups of $G$ and therefore
\begin{equation}
n_2(G)\mid p^r3^m=\frac{|G|}{2}\,.\nonumber
\end{equation}If $n_2(G)\neq\frac{|G|}{2}$\,, then $n_2(G)\leq\frac{|G|}{6}$, implying that
\begin{equation}
\frac{5}{12}\,|G|-2<\frac{|G|}{6}\,.\nonumber
\end{equation}This leads to $|G|<8$, a contradiction. Thus $n_2(G)=\frac{|G|}{2}$\,.

As in the proof of Lemma 2.2, d), $G'=NK$ is the unique subgroup of index $2$ in $G$. Clearly, all $3$-elements of $G$ are contained in $G'$ and so we have
\begin{equation}
n_3(G)=n_3(G')\leq\frac{3}{4}\,|G'|-1=\frac{3}{8}\,|G|-1\nonumber
\end{equation}by Theorem D, 1). Since the first four element orders of $G$ are $1$, $2$, $3$ and $p$, from (3) with $s=4$ we obtain
\begin{align*}
n_3(G)&>\frac{p-\frac{31}{12}}{p-3}\,|G|-\frac{p-1}{p-3}-\frac{p-2}{p-3}\,n_2(G)\\
&=\frac{6p-19}{12(p-3)}\,|G|-\frac{p-1}{p-3}\,.\nonumber
\end{align*}Thus
\begin{equation}
\frac{6p-19}{12(p-3)}\,|G|-\frac{p-1}{p-3}<\frac{3}{8}\,|G|-1,\nonumber
\end{equation}which is equivalent to $|G|<12$, a contradiction.

Assume next that $p=2$. Then $G$ possesses exactly $3^m$ Sylow $2$-subgroups. If $S\cong C_2^r\rtimes C_2$ is one of them, then
\begin{equation}
s_1(S)\leq s_1(D_8\times C_2^{r-2})=3\cdot 2^{r-1}-1\nonumber
\end{equation}by Theorem D, 4), and so $S$ has at least
\begin{equation}
2^{r+1}-1-(3\cdot 2^{r-1}-1)=2^{r-1}\nonumber
\end{equation}elements of order $\geq 4$. This implies that $G$ has at least
$2^{r-1}3^m$ elements of order $\geq 4$ and at most
\begin{equation}
2^{r+1}3^m-1-2^r(3^m-1)-2^{r-1}3^m=2^{r-1}3^m+2^r-1\nonumber
\end{equation}elements of order $2$. Similarly with the case $p\geq 5$, we get
\begin{equation}
n_3(G)\leq\frac{3}{8}\,|G|-1\nonumber
\end{equation}and since the first four element orders of $G$ are $1$, $2$, $3$, $4$, from (3) with $s=4$ we get
\begin{equation}
n_3(G)>\frac{17}{12}\,|G|-3-2n_2(G).\nonumber
\end{equation}Thus
\begin{equation}
\frac{17}{12}\,|G|-3-2n_2(G)<\frac{3}{8}\,|G|-1,\nonumber
\end{equation}implying that
\begin{equation}
n_2(G)>\frac{25}{48}\,|G|-1.\nonumber
\end{equation}Consequently, one obtains
\begin{equation}
\frac{25}{48}\,|G|-1<2^{r-1}3^m+2^r-1,\nonumber
\end{equation}which means $3^m<\frac{24}{13}$\,, a contradiction.\footnote{Note that the smallest example of such a finite group is $S_4$.}
\bigskip

Finally, we prove that if for a finite group $G$ we have $o(G)=\frac{31}{12}$\,, then $G\cong A_4$. We observe first that $G$ is solvable and all its proper quotients are su\-per\-sol\-va\-ble. Also, the equality $12\psi(G)=31|G|$ implies that both $2$ and $3$ divide $|G|$. Let $1=d_1<d_2<\ldots<d_r$ be the element orders of $G$. Then
\begin{equation}
\psi(G)=\sum_{i=1}^r d_in_{d_i}(G)=\sum_{i=1}^r d_i\varphi(d_i)n_{d_i}'(G),\nonumber
\end{equation}where $n_{d_i}'(G)$ denotes the number of cyclic subgroups of order $d_i$ in $G$, $\forall\, i=1,2,\ldots,r$. Since $\varphi(d_i)$ is even for all $d_i>2$, we infer that $\psi(G)\equiv\, 1\, ({\rm mod}\, 2)$. This shows that $|G|=4n$, where $n$ is odd and divisible by $3$.

If $G$ is supersolvable, then we get $G\cong G'\rtimes C_2^2$, as in the proof of Lemma 2.2, d). It follows that $G$ has a quotient $G_1$ of order $12$. Then $o(G_1)\leq\frac{31}{12}$\,, contradicting the fact that $o(X)>\frac{31}{12}$ for all supersolvable groups $X$ of order $12$.

If $G$ is not supersolvable, then it is a just non-supersolvable group and therefore $G\cong N\rtimes M$ with $M$ and $N$ as above. It is easy to see that the unique possibility to have $o(G)=\frac{31}{12}$ appears in {\bf Case 1} for $r=2$, i.e. for $G\cong C_2^2\rtimes C_3\cong A_4$, as desired. $\hspace{82mm}\qed$

\bigskip\noindent{\bf Proof of Theorem 1.2.} It follows immediately from Theorem 1.1 and Corollary 2.3.$\hspace{105mm}\qed$

\vspace*{3ex}\small

\hfill
\begin{minipage}[t]{5cm}
Marius T\u arn\u auceanu \\
Faculty of  Mathematics \\
``Al.I. Cuza'' University \\
Ia\c si, Romania \\
e-mail: {\tt tarnauc@uaic.ro}
\end{minipage}

\end{document}